\documentclass{amsart}
\usepackage{amsfonts,amssymb,amscd,amsmath,enumerate,verbatim,calc}
\usepackage[all]{xy}

\newcommand{\CM}{Cohen-Macaulay}

\newcommand{\Fb}{\mathbb{F} }

\newcommand{\wrt}{with respect to}
\newcommand{\B}{\mathcal{B} }

\newcommand{\bx}{\mathbf{x} }
\newcommand{\bff}{\mathbf{f} }
\newcommand{\Ic}{\mathcal{I} }
\newcommand{\n}{\mathfrak{n} }
\newcommand{\m}{\mathfrak{m} }
\newcommand{\M}{\mathfrak{M} }
\newcommand{\q}{\mathfrak{q} }

\newcommand{\R}{\mathcal{R} }
\newcommand{\V}{\mathcal{V} }
\newcommand{\Zc}{\mathcal{Z} }
\newcommand{\F}{\mathbb{F} }
\newcommand{\Hc}{\mathcal{H} }
\newcommand{\Z}{\mathbb{Z} }
\newcommand{\Pb}{\mathbb{P} }
\newcommand{\rt}{\rightarrow}
\newcommand{\xar}{\longrightarrow}
\newcommand{\ov}{\overline}

\newcommand{\wh}{\widehat }
\newcommand{\wt}{\widetilde }

\newcommand{\image}{\operatorname{image}}
\newcommand{\Om}{\Omega }
\newcommand{\Tor}{\operatorname{Tor}}
\newcommand{\codim}{\operatorname{codim}}
\newcommand{\Ass}{\operatorname{Ass}}

\newcommand{\cone}{\operatorname{cone}}
\newcommand{\ann}{\operatorname{ann}}
\newcommand{\e}{\operatorname{end}}

\newcommand{\reg}{\operatorname{reg}}
\newcommand{\cx}{\operatorname{cx}}
\newcommand{\projdim}{\operatorname{projdim}}
\newcommand{\Sp}{\operatorname{span}}
\newcommand{\CMa}{\operatorname{CM}}
\newcommand{\CMS}{\operatorname{\underline{CM}}}

\newcommand{\thick}{\operatorname{thick}}
\newcommand{\Hom}{\operatorname{Hom}}
\newcommand{\Ext}{\operatorname{Ext}}

\theoremstyle{plain}

\newtheorem{theorem}{Theorem}[section]
\newtheorem{corollary}[theorem]{Corollary}
\newtheorem{lemma}[theorem]{Lemma}
\newtheorem{proposition}[theorem]{Proposition}

\theoremstyle{definition}

\newtheorem{s}[theorem]{}
\newtheorem{remark}[theorem]{Remark}
\newtheorem{example}[theorem]{Example}

\theoremstyle{remark}

\begin{document}

\title[Ext functors]{Ext functors, support varieties and\\  Hilbert polynomials over complete intersection rings}
\author{Tony~J.~Puthenpurakal}
\date{\today}
\address{Department of Mathematics, IIT Bombay, Powai, Mumbai 400 076, India}

\email{tputhen@gmail.com}

 \begin{abstract}
Let $(A,\m)$ be a  complete intersection of dimension $d \geq 1$ and codimension $c \geq 1$. Let $I$ be an $\m$-primary ideal and let $M$ be a finitely generated $A$-module.
For $i \geq 1$ let $\psi_i^I(M)$ be the degree of the polynomial type function $n \rt \ell(\Ext^i_A(M, A/I^n))$. We show that for $j = 0, 1$ and  for all $i \gg 0$ we have $\psi_{2i +j}^I(M)$ is a constant and let $r_0^I(M)$ and $r_1^I(M)$ denote these constant values. Set $r^I(M) = \max\{ r_0^I(M), r_1^I(M) \}$. We show that $r^I(M)$ is an invariant of $I, A$ and the support variety of $M$. We set the degree of the zero polynomial to be $-\infty$. If $r^I(M) \leq 0$ then we show that $\reg G_I(\Omega^i(M))$ for $i \geq 0$ is bounded.  We give an application of this result to syzgetic Artin-Rees property of $M$. We also give several examples which illustrate our results.
\end{abstract}
 \maketitle
\section{introduction}
Let $(A,\m)$ be a \CM \ local ring  of dimension $d \geq 1$ and let $I$ be an $\m$-primary ideal. If $N$ is an $A$-module of finite length then $\ell(N)$ denotes its length. Let $M$ be a finitely generated $A$-module.  It is known that for $i \geq 1$, the function $n \rt \ell(\Ext_A^i(M, A/I^{n}))$ is of polynomial type say of degree $\psi_i^I(M)$, see \cite[Corollary 4]{Theo}.
When  $A$ is Gorenstein and not regular, $I = \m$ and $M$ is a non-free maximal \CM \ (= MCM) $A$-module then $\psi_i^\m(M) = d -1$ for all $i \geq 1$, see \cite[4.2]{PZ}.  We set the degree of the zero-polynomial to be $-\infty$.
When  $A$ is Gorenstein, $I $ a parameter ideal and $M$ is a  MCM $A$-module then $\psi_i^I(M) =  -\infty$ for all $i \geq 1$, see \cite[2.6]{PZ}.

It is easy to see that there exists an infinite  subset $X \subseteq \mathbb{N}$ such that $\psi_i^I(M)$  is a constant if $i \in X$. However it is not clear which infinite set occurs. The situation improves when $A$ is a complete intersection. We prove
\begin{theorem}\label{deg-const}
Let $(A,\m)$ be a complete intersection of dimension $d$ and codimension $c$. Let $I$ be an $\m$-primary ideal and let $M$ be a finitely generated $A$-module. Then for $j = 0, 1$ and  for all $i \gg 0$ we have $\psi_{2i +j}^I(M)$ is a constant.
\end{theorem}
\s\label{set}(with hypotheses as in \ref{deg-const})
For $j = 0, 1$ let $r_j^I(M) = \psi_{2i +j}^I(M)$ for $i \gg  0$. Set  $r^I_A(M) = \max\{ r_0^I(M), r_1^I(M) \}$.
If the ring $A$ is clear from the context then we drop the subscript $A$.
 When $A$ is a complete intersection there is a notion of support variety $\V_A(M)$ of a module $M$, see \cite{AB}. We show
\begin{theorem}\label{supp}
(with hypotheses as in \ref{set}) Let $N$ be another finitely generated $A$-module. If $\V_A(N) \subseteq \V_A(M)$ then $r^I(N) \leq r^I(M)$. In particular if $\V_A(N) = \V_A(M)$ then
$r^I(N) = r^I(M)$.
\end{theorem}

Next we consider the case when $r^I(M) = -\infty$. Let $G_I(A) = \bigoplus_{n \geq }I^n/I^{n+1}$ be the associated graded ring of $A$ \wrt \ $I$. Let  $G_I(M) = \bigoplus_{n \geq }I^nM/I^{n+1}M$ be the associated graded module of $M$ \wrt \ $I$ (considered as a $G_I(A)$-module). Let $\Omega^i(M)$ denote the $i^{th}$-syzygy of $M$ (for $i \geq 0$).
We show
\begin{theorem}\label{special}
(with hypotheses as above.) Let $r_I(M) = -\infty$. Also assume $M$ is a MCM $A$-module. Then there exists $\beta$ (depending only on $I$ and $M$) such that $\reg G_I(\Omega^i(M)) \leq \beta$ for all $i \geq 0$.
\end{theorem}

In the  section  fifteen we give bountiful examples of modules and ideals satisfying $r_I(M) = -\infty$.

It can be shown that $\V_A(I^n)$ is constant for $n \gg 0$, see \ref{as-ideal}. Set this constant value as $\V_A^\infty(I)$. In \ref{equi} we prove that $r_I(M) = - \infty$ if and only if
$\V_A(M)\cap\V_A^\infty(I) = \{ 0 \}$. The next case to consider is when $\dim \V_A(M) \cap \V_A^\infty(I) = 1$.
We show
\begin{theorem}\label{special-cx1}
(with hypotheses as above.)  Also assume $M$ is a MCM $A$-module. Assume $\dim \V_A(M) \cap \V_A^\infty(I) = 1$. Then there exists $\beta$ (depending only on $I$ and $M$) such that $\reg G_I(\Omega^i(M)) \leq \beta$ for all $i \geq 0$.
\end{theorem}
We need Theorem \ref{special} to prove Theorem \ref{special-cx1}.
An easy corollary to Theorems \ref{special} and \ref{special-cx1} is
\begin{corollary}\label{cx1}
(with hypotheses as above.)  Assume $\dim \V_A^\infty(I) \leq 1$. Also assume $M$ is a MCM $A$-module.  Then there exists $\beta$ (depending only on $I$ and $M$) such that $\reg G_I(\Omega^i(M)) \leq \beta$ for all $i \geq 0$.
\end{corollary}

Set $\V^{tot}_A(I) = \bigcup_{n \geq 1}\V_A(I^n)$. As $\V_A(I^n)$ is constant for $n \gg 0$, it follows that $\V^{tot}_A(I)$ is an algebraic variety. We show
\begin{theorem}\label{cx-mcm}
(with hypotheses as above.) There exists a number $\theta$ (depending only on $I$) such that if $M$ is any  MCM $A$-module with $\V_A(M) \cap \V_A^{tot}(I) = \{ 0 \}$  then
$\reg G_I(M) \leq \theta$.
\end{theorem}

Application:
The motivation for our application is a “uniform Artin–Rees” problem raised by Eisenbud and Huneke in \cite{EHu}, concerning uniform bounds of Artin–Rees type on free resolutions.
Let $M$ be a finitely generated $A$-module and let $(\Fb, \partial_{\bullet})$ be a  free
resolution of $M$ (not necessarily minimal).  We say $M$ is said to be \emph{ syzygetically Artin–Rees} with respect to
a family of ideals $\Ic$ if there exists a uniform integer $h$ such that for every $n \geq h$, for every $i \geq 0$ and for every  $I \in \Ic$,
\begin{equation}\label{sy-1}
  I^n\Fb_i \cap \image \partial_{i+1} \subseteq I^{n-h}\image \partial_{i+1}.
\end{equation}
This definition does not depend on the choice of resolution, see \cite[2.1]{S}. In \cite{EHu}, the authors show that given a Noetherian local ring $(A,\m)$, a module $M$ is syzygetically Artin–Rees with respect to an ideal $I$ under the hypotheses that $M$ has finite projective dimension and constant rank on the punctured spectrum. In a remarkable paper
\cite{AHS} it is shown that one can find a uniform $h$ for \emph{any} $d$-syzygy $A$-module $M$ and \emph{any} ideal $I$ of $A$ (here $d = \dim A$).

In this paper we study the following notion: Let $M$ be a finitely generated $A$-module and let $(\Fb, \partial_{\bullet})$ be a \emph{minimal} free
resolution of $M$.  We say $M$ is said to be \emph{ strongly syzygetically Artin–Rees} with respect to
an ideal $I$ if there exists an integer $h$ such that for every $n \geq h$ and for every $i \geq 0$,
\begin{equation}\label{sy-2}
  I^n\Fb_i \cap \image \partial_{i+1} =  I^{n-h}(I^h\Fb^i \cap \image \partial_{i+1}).
\end{equation}
It is not difficult to show that this definition does not depend on the choice of minimal resolution.

 Our question is trivial if $A$ is regular local.
 We prove the following general result on strong uniform Artin-Rees for syzygies.
\begin{theorem}\label{sunif}
Let $(A,\m)$ be a \CM \ local ring and let $I$ be an $\m$-primary ideal. Let $M$ be a finitely generated $A$-module and let $(\Fb_M, \partial_{\bullet})$ be a minimal free
resolution of $M$.  Assume $\{ \reg G_I(\Om^i(M)) \}_{ i \geq 0}$ is bounded. Then there is a positive number $h$ depending only on $I$ and $M$ such that
$I^n\Fb^i_M \cap \image \partial_{i+1} = I^{n-h}(I^h\Fb^i_M \cap \image \partial_{i+1})$ for all integers $n \geq h$ and for all $i \geq 0$.
\end{theorem}
Our results Theorems \ref{special}, \ref{special-cx1}, \ref{cx1} and \ref{cx-mcm} give bountiful examples of modules $M$ satisfying $\{ \reg G_I(\Om^i(M)) \}_{ i \geq 0}$ is bounded. This gives cases where Theorem \ref{sunif} applies.

 Here is an overview of the contents of this paper. In section two we discuss many preliminary facts that we need.
In section three we prove Theorem \ref{deg-const}. In section four we recall the notion of MCM approximations and prove some results that we need. In section five we discuss a Lemma that we need.
In section six we recall some facts about support varieties over a complete intersection. In section seven we prove Theorem \ref{supp}. In section eight we discuss a few preliminaries which hold when $r^I(M) = - \infty$. In section nine we discuss a base change result that we need. In section ten we discuss a technique of reducing complexity of a module that we need.
{In section eleven we discuss a few preliminary facts of the $A[It$-module $L^I(M) = \bigoplus_{n \geq 0}M/I^{n+1}M$. In section twelve we prove Theorem \ref{special}. In the next section we prove
Theorem \ref{special-cx1}.  In section fourteen we prove Theorem \ref{cx-mcm}. In section fifteen we construct bountiful examples which illustrate our results. In section sixteen we prove Theorem
\ref{sunif}. Finally in the appendix we discuss some results in elementary algebraic geometry that we need.
\section{Preliminaries}
In this section we assume  $(A,\m)$ is a Noetherian  local ring  of dimension $d \geq 1$. All modules considered (unless otherwise stated) will be finitely generated. Let $I$ be an $\m$-primary ideal. Let $M$ be an $A$-module. Fix $i \geq 1$ and consider the function $n \rt \ell(\Ext_A^i(M, A/I^{n}))$. This function is of polynomial type say of degree $\psi_i^I(M)$.

\s\label{red-mcm} Assume $A$ is \CM. Let $N  = \Om_A^1(M)$ be the first syzygy of $M$. Then for $i \geq 2$ we have $$\Ext_A^i(M, A/I^{n}) \cong  \Ext_A^{i -1}(N, A/I^{n}).$$ As we are interested in $\psi_i^I(M)$ for $i \gg 0$ it follows that we may assume that $M$ is a MCM $A$-module.

\s\label{red-rees} Assume $A$ is Gorenstein. Let $M$ be a MCM $A$-module. Considering the sequence $0 \rt I^n \rt A \rt A/I^n \rt 0$  and as $\Ext^i_A(M, A) = 0$ for $i \geq 1$, we obtain isomorphisms
$$\Ext_A^i(M, A/I^{n}) \cong  \Ext_A^{i +1}(M, I^{n}).$$
Let $\R = A[It]$ be the Rees algebra of $I$. We note that $\bigoplus_{n \geq 1}\Ext_A^{i +1}(M, I^{n})$ is a finitely generated $\R$-module. As its components have finite length, its dimension as a $\R$-module is $\leq d$.

We will need the following result:
\begin{lemma}\label{a-dim}
Let $S = \bigoplus_{n \geq 0}S_n$ be a standard graded ring over $S_0$. Assume Krull dimension of $ S_0$ is finite. Let $E = \bigoplus_{n \geq 0}E_n$ be a finitely generated graded $S$-module. Then $\dim_{S_0} E_n$ is constant for $n \gg 0$.
\end{lemma}
\begin{proof}
Let $x_1, \ldots, x_d \in S_1$ be such that $S = S_0[x_1, \ldots, x_d]$. We note that \\  $H^0_{S_+}(E)_n = 0$ for $n \gg 0$. Consider the map
\begin{align*}
  E_n &\rt (E_{n+1})^d \\
  m &\rt (x_1m, \cdots, x_dm)
\end{align*}
The kernel of this map is contained in  $H^0_{S_+}(E)_n$ which is zero for $n \gg 0$. It follows that $\dim_{S_0} E_n \leq \dim_{S_0} E_{n+1} $ for $n \gg 0$. As $\dim S_0$ is finite the result follows.
\end{proof}

An easy consequence of Theorem \ref{a-dim} is the following
\begin{corollary}\label{rachel}
Let $S_0$ be a Noetherian ring of finite Krull dimension. Let \\ $S = S_0[X_1, \ldots, X_c]$ be a polynomial ring with $\deg X_i = 2$. Let $E = \bigoplus_{n \geq 0}E_n$ be a finitely generated $S$-module.
Then for $j = 0, 1$,  we have  $\dim_{S_0} E_{2n + j}$ is constant for $n \gg 0$.
\end{corollary}
\begin{proof}
Let $E^{ev} = \bigoplus_{n \geq 0}E_{2n}$ with $E_{2n}$ sitting in degree $n$. Then $E^{ev}$ is a finitely generated $S_0[X_1, \ldots, X_c]$-module with $\deg X_i = 1$. By \ref{a-dim} it follows that $\dim_{S_0}E_{2n}$ is constant for $n \gg 0$.

For odd components the result follows by similarly considering $E^{o} = \bigoplus_{n \geq 0}E_{2n + 1}$ with $E_{2n + 1}$ sitting in degree $n$.
\end{proof}
\section{Proof of Theorem \ref{deg-const}}
In this section we give
\begin{proof}[Proof of Theorem \ref{deg-const}]
Without loss of any generality  we may assume that $A$ is complete.  Let $\R = A[Iu]$ be the Rees-algebra of $I$. By \ref{red-mcm} we may assume that $M$ is a MCM $A$-module. By \ref{red-rees} we have for $i \geq 1$
$$\Ext_A^i(M, A/I^{n}) \cong  \Ext_A^{i +1}(M, I^{n}).$$
By \cite[1.1]{P} we have
$\bigoplus_{i, n \geq 1}\Ext^i_A(M, I^n)$ is a finitely generated bigraded  \\ $\R[t_1, \ldots, t_c]$ module (here $\deg t_j = (0, 2)$ and $\deg yu^n = (n, 0)$ where $y \in I^n$).
Fix $i \geq 1$. Set $E_i = \bigoplus_{n \geq 1}\Ext^{i}_A(M, I^n)$. We note that $E_i $ is a finitely generated $\R$-module. Furthermore $\psi_i^I(M) = \dim_\R E_{i+1} - 1$.
Then $E = \bigoplus_{i \geq 1}E_i$ is a finitely generated $\R[t_1, \ldots, t_c]$-module with $\R$ sitting in degree zero and $\deg t_j = 2$ for all $j$. The result now follows from \ref{rachel}.
\end{proof}
\section{MCM Approximations}
In this section $(A,\m)$ is a Gorenstein local ring. 
We give a construction of MCM approximations of $M/(x_1, \ldots, x_r)M$ where $M$ is a MCM $A$-module and $x_1, \ldots, x_r$ is an $A$-regular sequence.

\s By $\CMS(A)$ we denote the stable category of MCM $A$-modules. Note by \cite[4.4.1]{Buchw} $\CMS(A)$ is a triangulated category with shift functor $\Omega^{-1}$ (the co-syzygy functor).
If $f \colon M \rt N$ is a morphism then by  $\cone(f)$ we denote the cone of $f$ in $\CMS(A)$; i.e., we have a triangle $M \xrightarrow{f} N \rt \cone(f) \rt \Omega^{-1}M$ in $\CMS(A)$.

\s We recall the notion of MCM approximations. Let $E$ be an $A$-module. By a MCM approximation of $E$ we mean an exact sequence
$$0 \rt Y \rt M \rt E \rt 0$$
 where $M$ is a MCM $A$-module and $\projdim Y$ is finite. Such sequences are not unique. However $M$ is a well-defined object in $\CMS(A)$. We usually say $M$ is a MCM approximation of $E$.

 \s \label{mcm-ex} We will need the following result, see \cite[1.4]{K}:
 let $0 \rt E_1 \rt E_2 \rt E_3 \rt 0$ be an exact sequence of $A$-modules. Then there exists an exact sequence
 $0 \rt M_1 \rt M_2 \rt M_3\rt 0$ where $M_i$ is a MCM-approximation of $E_i$.

 \s\label{construct} \emph{Construction:} Let $M$ be a MCM $A$-module. Let $x_1, \ldots, x_r$ be an $A$-regular sequence.
We construct MCM $A$-modules $V_i$ with $1 \leq i \leq r$ as follows:
\begin{enumerate}
  \item Let $V_1 = \cone(M \xrightarrow{x_1} M)$.
  \item For $i = 2, \ldots, r$ define $V_i = \cone(V_{i-1} \xrightarrow{x_i} V_{i-1})$.
\end{enumerate}
The main result of this section is
\begin{theorem}\label{mcm-V}
(with hypotheses as in \ref{construct}). For $i = 1, \ldots, r$;  $V_i$  is a MCM approximation of  $M/(x_1, \ldots, x_i)M$  in  $\CMS(A)$.
\end{theorem}

\s\label{tri}\emph{Triangulated category structure on $\CMS(A)$}. \\
The reference for this topic is \cite[4.7]{Buchw}. We first describe a
basic exact triangle. Let $f \colon M \rt N$ be a morphism in $\CMa(A)$. Note we have an exact sequence $0 \rt M \xrightarrow{i} Q \rt \Om^{-1}(M) \rt 0$, with $Q$-free. Let $C(f)$ be the pushout of $f$ and $i$. Thus we have a commutative diagram with exact rows
\[
  \xymatrix
{
 0
 \ar@{->}[r]
  & M
\ar@{->}[r]^{i}
\ar@{->}[d]^{f}
 & Q
\ar@{->}[r]^{p}
\ar@{->}[d]
& \Om^{-1}(M)
\ar@{->}[r]
\ar@{->}[d]^{j}
&0
\\
 0
 \ar@{->}[r]
  &N
\ar@{->}[r]^{i^\prime}
 & C(f)
\ar@{->}[r]^{p^\prime}
& \Om^{-1}(M)
    \ar@{->}[r]
    &0
\
 }
\]
Here $j$ is the identity map on $\Om^{-1}(M)$.
As $N, \Om^{-1}(M) \in \CMa(A)$  it follows that $C(f) \in \CMa(A)$.
Then the projection of the sequence
$$ M\xrightarrow{f} N \xrightarrow{i^\prime} C(f) \xrightarrow{-p^\prime} \Omega^{-1}(M)$$
in $\CMS(A)$  is a basic exact triangle. Exact triangles in $\CMS(A)$ are triangles isomorphic to a basic exact triangle.

Next we give a proof of
\begin{proof}[Proof of Theorem \ref{mcm-V}]
We prove the result by induction on $i$.

When $i = 1$ we note that by \ref{tri} there exists an exact sequence $0 \rt Q \rt V_1 \rt M/x_1M \rt 0$ with $Q$-free. It follows that $V_1$ is a MCM approximation of $M/x_1M$.

We assume the result when $i = a$ and prove for $i = a + 1$. By induction hypotheses we have an exact sequence
$$ 0 \rt Y \rt V_a\oplus F \rt M/(x_1, \ldots, x_a)M \rt 0, $$
where $\projdim Y$ is finite and $F$ is free. So we have an exact sequence
$$ 0 \rt Y/x_{a+1}Y \rt V_a/x_{a+1}V_a \oplus F/x_{a+1}F \rt M/(x_1, \ldots, x_{a+1})M \rt 0.$$
As $\projdim Y/x_{a+1}Y $ is finite it follows from \ref{mcm-ex} that a MCM approximation of $V_a/x_{a+1}V_a$ is a MCM approximation of  $M/(x_1, \ldots, x_{a+1})M $.
We note that by \ref{tri} there exists an exact sequence $0 \rt L \rt V_{a+1} \rt V_a/x_{a+1}V_a \rt 0$ with $L$-free. So $V_{a+1}$ is a MCM approximation of $V_a/x_{a+1}V_a$. The result follows.
\end{proof}

\s \label{mcm-fun} Let $\bx = x_1, \ldots, x_r$ be an $A$-regular sequence. Then the MCM approximation functor from $\CMS(A/(\bx)) \rt \CMS(A)$ is a triangulated functor, see \cite[p.\ 740]{BJM}.

\section{A Lemma}
In this section we prove the following result:
\begin{lemma}\label{mod-x}
Let $(A,\m)$ be a complete intersection ring and let $I$ be an $\m$-primary ideal. Let $M$ be a MCM $A$-module and let $\bx = x_1, \ldots, x_r$ be an $A$-regular sequence.
Let $D$ be a MCM-approximation of $M/\bx M$. Then $$r^I(D) = r^I(M).$$
\end{lemma}

We will need the following:
\begin{proposition}
\label{nil}
Let $S$ be a Noetherian ring of finite Krull dimension and let $E$ be an $A$-module. Let $x \in S$ such that $x^n E = 0$ for some $n \geq 1$.
Consider the exact sequence
$$ 0 \rt K \rt E \xrightarrow{x} E \rt C \rt 0. $$
Then $\dim K = \dim C = \dim E$.
\end{proposition}
\begin{proof}
We have nothing to prove if $E = 0$. So assume $E \neq 0$.
We note that $\dim K \leq \dim E$ and $\dim C \leq \dim E$. Set $T = S/\ann(E)$. Then $x$ is nilpotent in $T$. Let $P$ be a minimal prime of $E$ such that $\dim T/P = \dim E$. We note that
$x \in P$.

Claim-1: $K_P \neq 0$. (This will prove $\dim K = \dim E$).

If $K_P = 0$ then we have an injection $E_P \xrightarrow{x} E_P$. As $P$ is a minimal prime of $E$ the $A_P$-module $E_P$ is Artinian. So $E_P$ has finite length. Thus $E_P = xE_P$. Also $x \in PA_P$. By Nakayama lemma we get $E_P = 0$ which is a contradiction.

Claim-2: $C_P \neq 0$. (This will prove $\dim C = \dim E$).

If $C_P = 0$ then we have a surjection $E_P \xrightarrow{x} E_P$.  Thus $E_P = xE_P$. Also $x \in PA_P$. By Nakayama lemma  we get that $E_P = 0$ which is a contradiction.
\end{proof}

We now give
\begin{proof}[Proof of Lemma \ref{mod-x}]
By our construction \ref{mcm-V} it suffices to prove the result when $r = 1$. Set $x = x_1$ and let $0 \rt Q \rt V \rt M/xM \rt 0$ be a MCM approximation of $M/xM$ with $Q$-free.
Then note that $r^I(M/xM) = r^I(V)$.
The short exact sequence $0 \rt M \xrightarrow{x} M \rt M/xM \rt 0$ induces an exact sequence
\begin{align*}
  \Ext^{i-1}_A(M, A/I^n) &\xrightarrow{x} \Ext^{i-1}_A(M, A/I^n)  \rt \Ext^{i}_A(M/xM, A/I^n)  \\
  \rt \Ext^{i}_A(M, A/I^n) &\xrightarrow{x} \Ext^{i}_A(M, A/I^n)
\end{align*}
Let $\R = A[Iu]$ be the Rees algebra of $I$.
Consider the $\R$-modules \\ $E_i = \bigoplus_{n \geq 1} \Ext^{i}_A(M, A/I^n) $ and $C_i = \bigoplus_{n \geq 1} \Ext^{i}_A(M/xM, A/I^n)$. So for $i \geq 3$  we have an exact sequence of $\R$-modules
\[
E_{i-1} \xrightarrow{xu^0} E_{i-1} \rt C_i \rt E_i \xrightarrow{xu^0} E_i.
\]
We note that as $I$ is $\m$-primary some power of $x$ is in $I$. Also $E_i$ is a finitely generated $\R$-module. So some power of $x$ will annihilate $E_i$.
As $\psi^I_i(M) = \dim E_i -1$ it follows from \ref{nil} that for $i \geq 3$ we have
\[
\psi^I_i(M/xM) = \max \{ \psi^I_{i-1}(M), \psi^I_i(M) \}.
\]
The result follows from Theorem \ref{deg-const}.
\end{proof}
\section{Some Preliminaries on support varieties}
In this section we discuss a few preliminaries on support varieties that we need.

\s Let $(Q,\n)$ be a  local ring and let $\bff = f_1,\ldots, f_c \in \n^2$ be a regular sequence. Set $A = Q/(\bff)$ and denote its maximal ideal by $\m$.
The \emph{Eisenbud operators}, \cite{Eis}  are constructed as follows: \\
Let $\mathbb{F} \colon \cdots \rightarrow F_{i+2} \xrightarrow{\partial} F_{i+1} \xrightarrow{\partial} F_i \rightarrow \cdots$ be a complex of free
$A$-modules.

\emph{Step 1:} Choose a sequence of free $Q$-modules $\wt{F}_i$ and maps $\wt{\partial}$ between them:
\[
\wt{\mathbb{F}} \colon \cdots \rightarrow \wt{F}_{i+2} \xrightarrow{\wt{\partial}} \wt{F}_{i+1} \xrightarrow{\wt{\partial}} \wt{F}_i \rightarrow \cdots
\]
so that $\mathbb{F} = A\otimes\wt{\mathbb{F}}$.

\emph{Step 2:} Since $\wt{\partial}^2 \equiv 0 \ \text{modulo} \ (\mathbf{f})$, we may write  $\wt{\partial}^2  = \sum_{j= 1}^{c} f_j\wt{t}_j$ where
$\wt{t_j} \colon \wt{F}_i \rightarrow \wt{F}_{i-2}$ are linear maps for every $i$.

 \emph{Step 3:}
Define, for $j = 1,\ldots,c$ the map $t_j = t_j(Q, \mathbf{f},\mathbb{F}) \colon \mathbb{F} \rightarrow \mathbb{F}(-2)$ by $t_j = A\otimes\wt{t}_j$.
\s
The operators $t_1,\ldots,t_c$ are called Eisenbud's operator's (associated to $\mathbf{f}$) .  It can be shown that
\begin{enumerate}
\item
$t_i$ are uniquely determined up to homotopy.
\item
$t_i, t_j$ commute up to homotopy.
\end{enumerate}

\s Let $R = A[t_1,\ldots,t_c]$ be a polynomial ring over $A$ with variables $t_1,\ldots,t_c$ of degree $2$. Let $M, N$ be  finitely generated $A$-modules. By considering a free resolution $\mathbb{F}$ of $M$ we get well defined maps
\[
t_j \colon \Ext^{n}_{A}(M,N) \rightarrow \Ext^{n+2}_{R}(M,N) \quad \ \text{for} \ 1 \leq j \leq c  \ \text{and all} \  n,
\]
which turn $\Ext_A^*(M,N) = \bigoplus_{i \geq 0} \Ext^i_A(M,N)$ into a module over $R$. Furthermore these structure depend  on $\bff$, are natural in both module arguments and commute with the connecting maps induced by short exact sequences.

\s  Gulliksen, \cite[3.1]{Gull},  proved that if $\projdim_Q M$ is finite then
$\Ext_A^*(M,N) $ is a finitely generated $R$-module. 

\s We need to recall the notion of complexity of a module.
 Let $\beta_i^A(M) = \ell( \Tor^A_i(M,k) )$ be the $i^{th}$ Betti number of $M$ over $A$. The complexity of $M$ over $A$ is defined by
\[
\cx_A M = \inf\left \lbrace b \in \mathbb{N}  \left\vert \right.   \varlimsup_{n \to \infty} \frac{\beta^A_n(M)}{n^{b-1}}  < \infty \right \rbrace.
\]
If $A$ is a local complete intersection of $\codim c$ then $\cx_A M \leq c$. Furthermore all values between $0$ and $c$ occur.
Note that $\cx_A M \leq 1$ if and only if $M$ has bounded Betti numbers.

\s Assume $\projdim_Q M$ is finite. Note $\Ext^*_A(M,N) \otimes k$ is a finitely generated module over $S = R/\m R = k[t_1,\ldots,t_c]$.
Let $J = \ann_S \Ext^*_A(M, N) \otimes k$ be its annihilator. Let $\wt{k}$ be an algebraic closure of $k$.
Then  the variety $Var(J)$ in $\wt{k}^c$ is called the \emph{variety} of $M$ and $N$ and is denoted by $\V_A(M, N)$.

\s If $(A,\m)$ is a complete intersection ring of codimension $c$ then let $\wh{A}$ be the $\m$-adic completion of $A$. Then $\wh{A} = Q/\mathbf{q}$ for some regular local ring $(Q, \n)$ and ideal $\mathbf{q}$ generated by a regular sequence $\bff = f_1, \ldots, f_c \in \n^2$. If $M, N$ are $A$-module then let $\wh{M}, \wh{N}$ denote their completions. Set $\V_A(M,N) =  \V_{\wh{A}}(\wh{M}, \wh{N})$.\
It can be shown that this definition is independent of the Cohen-presentation of $\wh{A}$, see \cite[5.3]{AB}.

\s\label{prop-v}   We will need a few preliminary facts on support varieties see \cite[5.6]{AB}.

Define $\V_A(M) = \V_A(M, k)$. We have
\begin{enumerate}
 \item $\V_A(M, N)$ is a cone in $\wt{k}^c$.
  \item  $\V_A(M) = \V_A(M, M) = \V_A(k, M)$.
  \item $\V_A(M, N) = \V_A(N, M) = \V_A(M)\cap \V_A(N)$.
  \item  $\V_A(M, N) = \V_A(\Omega^i_A(M), \Omega^j(N))$ for all $i,j \geq 0$.
  \item $\V_A(M, N) = \{ 0 \}$ if and only if $\Ext^n_A(M, N) = 0$ for $n \gg 0$.
  \item $\dim \V_A(M, N) \leq 1$ if and only if $\Ext^i_A(M, N) \cong \Ext^{i+2}_A(M, N)$ for $i \gg 0$.
  \item If $0 \rt M_1 \rt M_2 \rt M_3 \rt 0$ is a short exact sequence of $A$-modules. For $\{i, j, h \} = \{1, 2, 3 \}$ we have
  $$ \V_A(M_h, N) \subseteq \V_A(M_i, N)\cup \V_A(M_j, N).$$
\end{enumerate}

\s \label{mod-x-new} Let $x \in \m \setminus \m^2$ be $A$-regular. If $A$ is a complete intersection of codimension $c$ then $A/(x)$ is also a complete intersection of codimension $c$. We show
\begin{proposition}\label{var-mod-x}
(with hypotheses as in \ref{mod-x-new}). Assume further $x$ is $M$-regular. Then
$$V_A(M) = V_{A/(x)}(M/xM). $$
\end{proposition}
\begin{proof}
Let $(\wh{A}, \wh{\m})$ be the completion of $A$. Then $(\wh{A}/x \wh{A}, \wh{\m}/(x))$
is the completion of $A/(x)$. Also if $\wh{A} = Q/(\bff)$ where $(Q,\n)$ is regular local and
$\bff = f_1, \ldots, f_c$ then note that $Q/(x)$ is also regular local and so
 $\wh{A/(x)} = Q/(x)/(\ov{\bff})$ is also a Cohen-presentation of $A/(x)$. We note that $x$
 is also $\wh{M}$-regular. Thus we can assume that $A$ is complete. We note that $\codim A/(x) = \codim A$.

 Let $\F$ be a resolution of $M$. Then  $\F/x\F$ is a resolution of $M/xM$ as an $A/(x)$-module. It readily follows that the Eisenbud operators on $\F$ induce the Eisenbud-operators on $\F/x\F$. We note that there is an isomorphism of complexes
 \[
 \Hom_{A/(x)}(\F \otimes_A A/(x), k) \cong \Hom_A(\F, \Hom_{A/(x)}(A/(x), k) = \Hom_A(\F, k).
 \]
 The Eisenbud operators on both the complexes are the same. It follows that we have an isomorphism of $k[t_1, \ldots, t_c]$-modules $\Ext^*_A(M,k)$ and $\Ext^*_{A/(x)}(M/xM, k)$. The result follows.
\end{proof}

\begin{remark}\label{proj}
In our examples we will consider for convenience  $A$ to be complete with algebraically
closed residue field $k$. We will also need to do some elementary algebraic geometry in our examples. It is easier to work with varieties in projective spaces than cones in affine space. As
$\V_A(M, N)$ is a cone in $k^{c}$, it will correspond to a variety $\Zc_A(M, N)$ in $\Pb^{c-1}_k$.
\end{remark}

\section{Proof of Theorem \ref{supp}}
In this section we give a proof of Theorem \ref{supp}. We need a few preliminaries. In this section $(A,\m)$ is a complete intersection of codimension $c$ and $I$ is an $\m$-primary ideal.
\begin{lemma}
\label{thick} Let $M$ be a MCM $A$-module. Consider $\thick(M)$ the thick subcategory in $\CMS(A)$ generated by $M$. If $N \in \thick(M)$ then $r_I(N) \leq r_I(M)$.
\end{lemma}
\begin{proof}
We prove the result in a few steps.

Claim-1 $r^I(E) = r^I(\Omega(E)) = r^I(\Omega^{-1}(E))$.
The first assertion follows from \ref{red-mcm}. The second follows from the first as $\Omega(\Omega^{-1}(M)) \cong M$ in $\CMS(A)$.

Claim-2: If $D \rt N \rt E \rt \Om^{-1}(D)$ is a triangle in $\CMS(A)$ then
$$ r^I(E) \leq \max \{ r^I(N), r^I(D) \}.$$
By the construction of triangles in $\CMS(A)$, there exists an exact sequence of $A$-modules
$$0\rt N \rt E\oplus F \rt \Om^{-1}(D) \rt 0 \quad \text{where} \ F \ \text{is a free $A$-module.} $$
It follows that for all $i \geq 1$ there exists an exact sequence
\[
\rt \Ext^i_A(\Om^{-1}(D), A/I^n) \rt \Ext^i_A(E, A/I^n) \rt \Ext^i_A(N, A/I^n) \rt
\]
It follows that $$\psi^I_i(E) \leq \max \{ \psi^I_i(N), \psi^I_i(\Om^{-1}(D)) \}.$$
So we have
$$r_I(E) \leq \max \{ r^I(N), r^I(\Om^{-1}(D)\}.$$
The result follows from Claim-1.

Claim-3: If $U$ is a direct summand of $V$ then $r^I(U) \leq r^I(V)$.

It is elementary to note that $\psi^I_i(U) \leq \psi^I_i(V)$ for all $i \geq 1$. The result follows.

The result now follows from the construction of thick subcategory generated by a module $M$.
\end{proof}

We now give
\begin{proof}[Proof of Theorem \ref{supp}]
By \ref{red-mcm} and \ref{prop-v} it follows that we can assume $M, N$ to be MCM $A$-modules.
  Let $\bx = x_1, \ldots, x_d \in \m\setminus \m^2$ be a maximal regular sequence in $A$. As $M, N$ are MCM $A$-modules we get that  $\bx$ is $M$ and  $N$-regular sequence. Set $B = A/(\bx)$. By \ref{var-mod-x} it follows that
  $\V_A(M) = \V_B(M/\bx M)$ and $\V_A(N) = \V_B(N/\bx N)$. We have $\V_B(N/\bx N) \subseteq \V_B(M/\bx M)$. So by \cite[5.6, 5.12]{CI}, it follows that $N/\bx N \in \thick_B(M/\bx M)$.
  Consider the MCM approximation functor $\CMS(B) \rt \CMS(A)$. Let $X_M( \text{resp.} \ X_N)$ be MCM-approximations of $M/\bx M( \text{resp.} \ N/\bx N)$. Then $X_N \in \thick_A(X_M)$.
  By \ref{construct} and \ref{mcm-V}, it follows that $X_M \in \thick(M)$. So $X_N \in \thick(M)$. So by Lemma \ref{thick}, we get $r^I(X_N) \leq r^I(M)$. By \ref{mod-x}, we get $r^I(X_N) = r^I(N)$. The result follows.
\end{proof}

\section{some preliminaries when $r^I(M) = - \infty$}
In this section we give some consequences of the assumption that $r^I(M) = -\infty$. Throughout this section $A$ is a local complete intersection of codimension $c$. We will assume that $M$ is  a MCM $A$-module.

\s \label{as-ideal} Let $J$ be any ideal in $A$. In \cite[7.8]{P} we showed that $\V_A(k, J^n)$ is constant for $n \gg 0$.
By \ref{prop-v} we have $\V_A(k, J^n) = \V_A(J^n, k) = \V_A(J^n)$ for all $n$. We set $\V^\infty_A(J) = $ the constant value of $\V_A(J^n) $ for all $n \gg 0$.

Next we give equivalent conditions for $r^I(M) = -\infty$.
\begin{theorem}\label{equi}
Let $(A,\m)$ be a complete intersection of dimension $d \geq 1$ and  codimension $c \geq 1$. Let $M$ be a MCM $A$-module. The following assertions are equivalent:
\begin{enumerate}[\rm (i)]
  \item $r^I(M) = -\infty$.
  \item There exists $i_0$ such that for $i \geq i_0$ there exists $n(i)$ (depending on $i$) such that $\Ext^i_A(M, A/I^n) = 0$ for $n \geq n(i)$.
  \item There exists $i_0$ and $n_0$ such that $\Ext^i_A(M, A/I^n) = 0$  for $i \geq i_0$ and $n \geq n_0$.
  \item $\V^\infty_A(I) \cap \V_A(M) = \{ 0 \}$.
  \item Assume $\V_A(I^n) = \V^\infty_A(I)$ for $n \geq t$. Then $\Tor^A_i(M, A/I^n) = 0$ for all $i > 0$ and $n \geq t$.
\end{enumerate}
\end{theorem}
\begin{proof}
We will use the fact that for any ideal $J$ we have $\V_A(J) = \V_A(A/J)$.
  (i) $\Leftrightarrow$ (ii).  This follows from definition and Theorem \ref{deg-const}.

  (iii) $\implies$ (ii). This is clear.

  (ii) $\implies$ (iii). We may assume that $A$ is complete. Let $\R = A[Iu]$ denote the Rees algebra of $I$.  We may also assume that $M$ is MCM $A$-module. Then
  $E = \bigoplus_{i, n \geq 1} \Ext^i_A(M, A/I^n) = \bigoplus_{i, n \geq 1} \Ext^{i +1}_A(M, I^n)$ is a  finitely generated bigraded $\R[t_1, \ldots, t_c]$-module with $\deg t_i = (0, 2)$ and $\deg yu^n = (n,0)$ where $y \in I^n$, see \cite[1.1]{P}.
  Set $E^{ev} = \bigoplus_{i, n \geq 1} \Ext^{2i}_A(M, A/I^{n})$ is a  finitely generated bigraded $\R[t_1, \ldots, t_c]$-module with $\deg t_i = (0, 1)$ and $\deg yu^n = (n,0)$ where $y \in I^n$.
  By \cite[3.4]{W} there exists $(i_0, n_0)$ such that $\Ass_A E^{ev}_{i,n} = \Ass_A E^{ev}_{i_0, n_0}$ for all $(i,n) \geq (i_0, n_0)$. It follows that for all $(i,n) \geq (i_0, n_0)$ we have
  $\Ext^{2i}_A(M, A/I^n) = 0$. For odd values of $i$ the result follows by considering $E^{odd} = \bigoplus_{i, n \geq 1} \Ext^{2i + 1}_A(M, A/I^{n})$.

  (iii) $\implies$ (iv). It follows from \ref{prop-v}  that for $n \geq n_0$ we have $\V_A(I^n) \cap \V_A(M) = \{ 0 \}$. The result follows from \ref{as-ideal}.

  (iv) $\implies$ (ii) Say for $n \geq n_0$ we have $\V_A(I^n) = \V_A^\infty(I)$. Then we have for $n \geq n_0$ we have $\V_A(I^n)\cap \V_A(M) = \{ 0 \}$. The result follows from \ref{prop-v}.

  (v) $\implies$ (iv) It follows from \cite[III, IV]{AB} that $\V_A(I^n) \cap \V_A(M) = \{ 0 \}$. 

  (iv) $\implies$ (v)  The result follows from \cite[III, IV]{AB} and \cite[4.9]{AB}.
\end{proof}

As an application of the above result we show:
\begin{proposition}
\label{mod-sup}
Let $(A,\m)$ be a complete intersection of dimension  and let $I$ be an $\m$-primary ideal. Let $M$ be a MCM $A$-module and let $x \in I$ be $A\oplus \Omega^i(M)$-superficial for all $i \geq 0$.
Set $B = A/(x)$, $J = I/(x)$ and $N = M/xM$. Assume $r^I_A(M) = -\infty$.
Then
$r^J_B(N) = -\infty$.
\end{proposition}
\begin{proof}
As $x$ is $I$-superficial we have $(I^{n+1} \colon x) = I^n$ for $n \gg 0$ (say $n \geq n_0$). So we have an exact sequence
$$ 0 \rt A/I^n \xrightarrow{\alpha_n} A/I^{n +1} \rt B/J^{n+1} \rt 0,$$
where $\alpha_n(a + I^n) = xa + I^{n+1}$. So we get a exact sequences for all $i \geq 1$,
$$\cdots \rt \Ext^i_A(M, A/I^{n+1}) \rt \Ext^i_A(M, B/J^{n+1}) \rt \Ext^{i+1}_A(M, A/I^n) \rt \cdots.$$
As $r^I_A(M) = -\infty$ we obtain by \ref{equi}, $\Ext^i_A(M, A/I^{n}) = 0$ for $i,n \gg 0$.
So we have  $ \Ext^i_A(M, B/J^{n+1})  = 0$  for $i,n \gg 0$. It remains to note that
  $ \Ext^i_A(M, B/J^{n+1}) =  \Ext^i_B(N, B/J^{n+1})$. The result follows from \ref{equi}.
\end{proof}
\begin{remark}
If the residue field of $A$ is uncountable then we can choose $x \in I$ which  is  $A\oplus \Omega^i(M)$-superficial for all $i \geq 0$, see \cite[2.2]{Pci}.
\end{remark}
\section{Base-change}
To prove  some results later we need to do base-change to a ring $C$ with uncountable residue field which is also algebraically closed. Furthermore we also assume that $C$ is complete.

\s \label{gono} Let $(A,\m)$ be a complete intersection of co-dimension $c$ with residue field $k$.
Let $K = \ov{k((X))}$. Then by \cite[Appendice, sect 2]{BourIX} there exists a flat extension $(B, \n)$ with $\m B = \n$ and the residue field of $B$ is $K$. We note that $B$ is also a complete intersection of co-dimension $c$. Take $(C, \q)$ to be the completion of $B$.

\begin{proposition}\label{bc}
(with hypotheses as in \ref{gono}). Let $M$ be a finitely generated $A$-module and let $I$
be an $\m$-primary ideal. Then
\begin{enumerate}[\rm (1)]
\item
If $r^I_A(M) = -\infty$ then $r^{IC}_C(M\otimes_A C) = -\infty$.
\item
If  $\dim \V_A(M) \cap \V_A^\infty(I) \leq 1$ then $\dim \V_C(M\otimes_A C) \cap \V_C^\infty(IC) \leq 1$
\end{enumerate}
\end{proposition}
\begin{proof}
(1)By \ref{equi} there exists $i_0, n_0$ such that $\Ext^{i}_A(M, A/I^n) = 0$ for $i \geq i_0$ and $n \geq n_0$. It follows that $\Ext^{i}_C(M\otimes_A C, C/I^nB) = 0$ for $i \geq i_0$ and $n \geq n_0$. So by \ref{equi} we get $r^{IC}_C(M\otimes_A C) = -\infty$.

(2) Sat $\V_A(I^n)$ is constant for $n \geq n_0$. Fix $n \geq n_0$. Then by \ref{prop-v} we have
$\Ext^{i+2}_A(M, A/I^n) \cong \Ext^{i}_A(M, A/I^n)$ for all $i \gg 0$. It follows that
$\Ext^{i+2}_C(M \otimes_A C, C/I^nC) \cong \Ext^{i}_C(M \otimes_A C, C/I^nC)$ for all $i \gg 0$. The result follows from \ref{prop-v}.
\end{proof}
\section{Reducing complexity of a module}
\s \label{red-cx} In this section $(A, \m)$ is a complete, complete intersection of codimension $c$ with algebraically closed residue field. Let $M$ be a finitely generated $A$-module.
Let $\F$ be a minimal resolution of $M$. Let $t_1, \ldots, t_c$ be the Eisenbud operators acting on $\F$.
Then $\Ext^*_A(M, k)$ is a finitely generated $S = k[t_1, \ldots, t_c]$-module.
The following result is certainly known. We sketch a proof for the convenience of the reader.
\begin{theorem}\label{red-cx-thm}(with hypotheses as in \ref{red-cx}). Set $M_i =\Omega^i_A(M)$. Assume $r = \cx_A M \geq 2$. Suppose $L_1, \ldots, L_s$ be proper linear-subvarieties of $k^c = S_2$. Then
\begin{enumerate}[\rm (1)]
  \item  There exists exact sequence $0 \rt K_i \rt M_{i+2} \rt M_i \rt 0$ for $i \gg 0$. Here $\cx K_i = r -1$ and $\Omega^1_A(K_i) = K_{i+1}$.
  \item $V(K_i) \subseteq V(M)$ and $V(K_i)\cap L_j$ is a proper linear subvariety of $L_j$ for all $j$.
\end{enumerate}
\end{theorem}
\begin{proof}
Set $E = \Ext^*_A(M, k)$. Set $\M$ to be maximal homogeneous ideal of $S$. Let
\[
\Ass^*(E) = \{ P \in \Ass_S E \mid P \neq \M \}.
\]
For $P \in \Ass^*(E)$ we note that $P \cap S_2$ consists of homogeneous linear polynomials in $t$. Furthemore $V(P\cap S_2) \neq S_2$.
Choose $u =\sum_{i = 1}^{c} \alpha_i t_i$ such that $u \notin P \cap S_2$ for all $P \in \Ass^*(E)$ and $L_j \nsubseteq H$ for all $j$ where $H = \{ (a_1, \ldots, a_c) \in S_2 \mid
\sum_{i=1}^{c}a_i\alpha_i = 0 \}$.

We note that $u$ is $E$-filter regular. So $(0 \colon_E u)$ has finite length. Then we have injections $\Ext^i_A(M, k) \xrightarrow{u_i} \Ext^{i+2}_A(M, k)$ for all $i \gg 0$.
Dualizing we get surjective maps $\Tor^A_{i+2}(M, k) \rt \Tor^A_i(M, k)$ for all $i \gg 0$. Take a lift $\beta_i$ in $A$ of each $\alpha_i$. Set $v = \sum_{i=1}^{c} \beta_i t_i$. Then $v  \colon \F(+2) \rt \F$ is a chain map. Furthermore by Nakayama lemma we get that $\F_{i+2} \xrightarrow{v} \F_i$ is surjective for all $i \gg 0$. So we get exact sequences for all $i \gg 0$,
\begin{enumerate}[\rm (a)]
  \item  $0 \rt G_i \rt \F_{i+2} \rt \F_{i} \rt 0$ for all $i \gg 0$, say from $i \geq i_0$.
  \item $0 \rt K_i \rt M_{i+2} \rt M_i \rt 0$ for $i \geq i_0$.
  \item  An exact sequence $0 \rt \Tor^A_n(K_i, k) \rt \Tor^A_{n+2}(M_{i+2}, k) \rt \Tor^A_n(M_i, k)$ for $n \geq 1$
\end{enumerate}
We note that $\mathbb{G}[i_0] \rt K_{i_0}$ is a projective resolution.  It is also minimal as $\F$ is minimal. So $\Omega^1_A(K_i) = K_{i+1}$ for $ i \geq i_0$. By (c) it follows that $\cx_A K_{i_0} = r -1$.

We have an exact sequence $$0 \rt \Ext^n_A(M,k) \xrightarrow{u_n} \Ext^{n+2}_A(M,k) \rt \Ext^{n-i_0}(K_{i_0}, k) \rt 0,$$ for $n \geq i_0 + 1$. We not that $\Ext^{\geq 1}_A(K_{i_0}, k) \subseteq E/uE[i_0]$.  It follows that $\V(K_{i_0}) \subseteq \V(M) \cap V(u)$ We note that $V(u)\cap L_j$ is a proper linear sub-variety of $L_j$. The result follows.
\end{proof}
\section{Some Properties of $L^{I}(M)$}\label{Lprop}

In this section we collect some of  the properties of $L^{I}(M) = \bigoplus_{n\geq 0}M/I^{n+1}M $ which we proved in \cite{Pu5}. Throughout this section
$(A,\m)$ is a \CM \ local ring with infinite residue field, $M$ is a \emph{\CM }\ module of dimension $r \geq 1$ and $I$ an $\m$-primary ideal.

\s \label{mod-struc} Set $\R = A[It]$;  the Rees Algebra of $I$.
The $A$-module $L^I(M)$ can be given an $R(I)$-module structure as follows.
The Rees ring $\R$ is a subring of $A[t]$ and so $A[t]$ is an $\R$-module.
Therefore $M[t] = M\otimes_A A[t]$ is an $\R$-module.
Let $R(I,M) = \bigoplus_{n \geq 0}I^nM$ be the Rees module of $M$ \wrt \ $I$. It is a finitely generated $\R$-module.
 The exact sequence
\[
0 \xar R(I,M) \xar M[t] \xar L^I(M)(-1) \xar 0
\]
defines an $\R$-module structure on $ L^I(M)(-1)$ and so on $ L^I(M)$.

 Note that $L^I(M)$ is \emph{not a finitely generated} $\R(I)$-module.

\s Let $\M$ denote the maximal homogeneous ideal of $\R$. Let $H^{i}(-) = H^{i}_{\M}(-)$ denote the $i^{th}$-local cohomology functor \wrt \ $\M$.
 Recall a graded $\R $-module $V$ is said to be
\textit{*-Artinian} if
every descending chain of graded submodules of $V$ terminates. For example if $E$ is a finitely generated $\R$-module then $H^{i}(E)$ is *-Artinian for all
$i \geq 0$.

\s \label{zero-lc} In \cite[4.7]{Pu5} we proved that
\[
H^{0}(L^I(M)) = \bigoplus_{n\geq 0} \frac{\wt{I^{n+1}M}}{I^{n+1}M}.
\]
Here $\widetilde{KM}$ denotes the Ratliff-Rush closure of $K$ \wrt \  $M$.  Recall
\[
\widetilde{KM} = \bigcup_{i \geq 1}K^{i+1}M \colon K^i.
\]
\s \label{Artin}
For $L^I(M)$ we proved that for $0 \leq i \leq  r - 1$
\begin{enumerate}[\rm (a)]
\item
$H^{i}(L^I(M))$ are  *-Artinian; see \cite[4.4]{Pu5}.
\item
$H^{i}(L^I(M))_n = 0$ for all $n \gg 0$; see \cite[1.10 ]{Pu5}.
\item
 $H^{i}(L^I(M))_n$  has finite length
for all $n \in \mathbb{Z}$; see \cite[6.4]{Pu5}.
\item
For $0 \leq i \leq r-1$
there exists polynomial $q_i(z) \in \mathbb{Q}[z]$ such that $q_i(n) = \ell(H^i(L^I(M))_n)$ for all $n \ll 0$.

\end{enumerate}
\s \label{I-FES} The natural maps $0\rt I^n/I^{n+1} \rt M/ I^{n+1}\rt M/I^n \rt 0 $ induce an exact
sequence of $\R$-modules
\begin{equation}
\label{dag}
0 \xar G_{I}(M) \xar L^I(M) \xrightarrow{\Pi} L^I(M)(-1) \xar 0.
\end{equation}
We call (\ref{dag}) \emph{the first fundamental exact sequence}.  We use (\ref{dag}) also to relate the local cohomology of $G_I(M)$ and $L^I(M)$.

\s If $E = \bigoplus_{n \in \Z}E_n$ is a $\R$-module is non-zero, set $\e E = \max \{ n \mid E_n \neq 0 \}$.  Set $\e 0 = -\infty$.
Set $b_i^I(M) = \e H^i(L^I(M))$ for $i = 0,\ldots, r-1$. Set
$$\reg L^I(M) = \max\{ b_i^I(M) + i \mid 0\leq i \leq r -1 \}.$$
Also set
$$\reg^{\geq 1} L^I(M) = \max\{ b_i^I(M) + i \mid 1\leq i \leq r -1 \}.$$

\s\label{compare} Set $a_i(G_I(M)) = \e H^i(G_I(M))$. Now assume that $M$ is a MCM $A$-module. We have a surjection $A^r \rt M \rt 0$, yields a surjection $G_I(A^r) \rt G_I(M) \rt 0$. So we have $a_d(G_I(M)) \leq a_d(G_I(A))$. It follows from this and \ref{I-FES} that
\[
\reg G_I(M) \leq \max \{ \reg L^I(M) + 2, a_d(G_I(A)) + d \}.
\]

\s \label{dim1} Assume  $\dim A =  1$ and $M$ is a MCM $A$-module. Then note $b_0^I(M) \leq a_1(G_I(M)) - 1 \leq a_1(G_I(A)) -1$. By \ref{compare} it follows that $\reg G_I(M)$ is bounded by
$a_1(G_I(A)) + 1$. This bound only depends on $I$ and is independent of $M$.

\s \label{II-FES} Let $x$ be  $M$-superficial \wrt \ $I$, i.e., $(I^{n+1}M \colon x) = I^nM$ for all $n \gg 0$. Set  $N = M/xM$ and $u =xt \in \R_1$. Notice $L^I(M)/u L^I(M) = L^I(N)$.
For each $n \geq 1$ we have the following exact sequence of $A$-modules:
\begin{align*}
0 \xar \frac{I^{n+1}M\colon x}{I^nM} \xar \frac{M}{I^nM} &\xrightarrow{\psi_n} \frac{M}{I^{n+1}M} \xar \frac{N}{I^{n+1}N} \xar 0, \\
\text{where} \quad \psi_n(m + I^nM) &= xm + I^{n+1}M.
\end{align*}
This sequence induces the following  exact sequence of $\R$-modules:
\begin{equation}
\label{dagg}
0 \xar \B^{I}(x,M) \xar L^{I}(M)(-1)\xrightarrow{\Psi_u} L^{I}(M) \xrightarrow{\rho^x}  L^{I}(N)\xar 0,
\end{equation}
where $\Psi_u$ is left multiplication by $u$ and
\[
\B^{I}(x,M) = \bigoplus_{n \geq 0}\frac{(I^{n+1}M\colon_M x)}{I^nM}.
\]
We call (\ref{dagg}) the \emph{second fundamental exact sequence. }

\s \label{long-mod} Notice  $\lambda\left(\B^{I}(x,M) \right) < \infty$. A standard trick yields the following long exact sequence connecting
the local cohomology of $L^I(M)$ and
$L^I(N)$:
\begin{equation}
\label{longH}
\begin{split}
0 \xar \B^{I}(x,M) &\xar H^{0}(L^{I}(M))(-1) \xar H^{0}(L^{I}(M)) \xar H^{0}(L^{I}(N)) \\
                  &\xar H^{1}(L^{I}(M))(-1) \xar H^{1}(L^{I}(M)) \xar H^{1}(L^{I}(N)) \\
                 & \cdots \cdots \\
               \end{split}
\end{equation}

We will need the following consequence of (\ref{longH}).
\begin{lemma}\label{reg1}
Let $M$ be \CM \ $A$-module of dimension $r \geq 2$. Let $x \in I$ be $M$-superficial. Let $N = M/xM$. Then $\reg^{\geq 1}L^I(M) \leq \reg L^I(N)$.
\end{lemma}
\begin{proof}
Let $0 \leq i \leq r-2$. If $H^i(L^I(N))_n = 0$ for $n \geq c$ then by \ref{longH} we have injections $H^{i+1}(L^I(M))_{n -1} \rt H^{i+1}(L^I(M))_n$ for $n \geq c$. As $H^{i+1}(L^I(M))_n = 0$
for $n \gg 0$ it follows that $b_{i+1}^I(M) \leq b^I_{i}(N) - 1$. The result follows.
\end{proof}

\s\label{Veronese} One huge advantage of considering $L^I(M)$ is that it behaves well \wrt \ the Veronese functor. Notice
\[
\left(L^I(M)(-1)\right)^{<t>} = L^{I^t}(M)(-1) \quad \text{for all} \ t \geq 1.
\]
Also note that $\R(I)^{<t>} = \R(I^t)$ and that $(\M_{\R(I)})^{<t>} = \M_{\R(I^t)}$. It follows that for all $i \geq 0$
\[
\left(H^i_{\M_{\R(I)}}(L^I(M)(-1)\right)^{<t>} \cong H^i_{\M_{\R(I^t)}}(L^{I^t}(M)(-1).
\]

We will need the following result.
\begin{lemma}
\label{h0}
Let $M$ be \CM \ $A$-module of dimension $r \geq 2$. Let $x \in I$ be $M$-superficial. Let $N = M/xM$. Set $b = \e H^0(L^I(N))$. Let $t = \e H^0(L^{I^m}(M))$ where $m > b$.
Then $$\e H^0(L^I(M)) \leq m(t+2) - 1.$$
\end{lemma}
\begin{proof}
By (\ref{longH}) and \ref{zero-lc} we get surjections for $n > b$.
\[
\frac{\wt{I^n M}}{I^nM} \rt \frac{\wt{I^{n+1}M}}{I^{n+1}M}.
\]
We have by \ref{zero-lc}
\[
\wt{I^{m(t+2)}M} = I^{m(t+2)}M.
\]
It follows that $\wt{I^jM} = I^jM$ for all $j \geq m(t+2)$. The result follows.
\end{proof}

\section{proof of Theorem  \ref{special}}
In this section we prove Theorem \ref{special}. The proof consists of first proving the case for large powers of $I$ and then later concluding it for $I$.
We first prove the following:
\begin{theorem}
\label{first}Let $(A,\m)$ be a complete intersection of dimension $d \geq 1$ and codimension $c \geq 1$. Let $I$ be an $\m$-primary ideal and let $M$ be an MCM $A$-module with $r^I(M) = -\infty$.
Let $\V(I^m) = \V(I^{s})$ for $m \geq s$. Fix $m \geq s$. Then
$\{ \reg L^{I^m}(\Omega^r(M) \}_{r \geq 0}$ is bounded.
\end{theorem}
\begin{proof}
Set $J = I^m$.
We induct on complexity $\cx M$ of $M$. If $\cx M = 0$ then $M$ is free and so we have nothing to show. If $\cx M = 1$ then $M$ has periodic resolution, So again we have nothing to show.
So assume that $\cx M  = t \geq 2$ and we have proved result when complexity of the module is $t -1$. Set $M_i = \Om^{i}(M)$ for $i \geq 0$. We have by \ref{red-cx-thm} there exists an exact sequence
$$ 0 \rt K_i \rt M_{i+1} \rt M_i \rt 0, \quad \text{for} \ i \geq  i_0.$$
We also have $K_{i+1} = \Om(K_i)$ for all $ i \geq i_0$ and $\cx_A K_{i_0} = \cx M - 1 = t -1$.
We note that for all $n$ we have $\Tor^A_i(M, A/I^n) = 0$ for $i \geq 1$, see \ref{equi}. We have for all $j \geq i_0 + 1$
an exact sequence
\[
0 \rt L^J(K_i) \rt L^J(M_{i+2}) \rt L^J(M_i) \rt 0
\]
and $\Tor^A_i(K, A/I^n) = 0$. It follows that $r^I(K) = - \infty$. So by induction hypotheses $\{ \reg L^{J}(\Omega^r(K_{i_0})) \}_{r \geq 0}$ is bounded (say by $\alpha$).
We have $\reg L^J(M_{j_0 + 2}) \leq \max \{ \alpha, \reg L^J(M_{j_0}) \}.$
We also have
$$ \reg L^J(M_{j_0 + 4}) \leq \max \{ \alpha, \reg L^J(M_{j_0 + 2}) \} \leq \max \{ \alpha, \reg L^J(M_{j_0}) \}.$$
Iterating we obtain
$$ \reg L^J(M_{j_0 + 2n})  \leq \max \{ \alpha, \reg L^J(M_{j_0}) \} \quad \text{for all } \ n \geq 1.$$
Similarly we obtain
$$ \reg L^J(M_{j_0 + 2n+1})  \leq \max \{ \alpha, \reg L^J(M_{j_0 + 1}) \} \quad \text{for all } \ n \geq 1.$$
The result follows.
\end{proof}

We now give
\begin{proof}[Proof of Theorem \ref{special}]
By \ref{gono} we may assume that $A$ is complete
and that the residue field of $A$ is uncountable and algebraically closed. We prove the result by induction on $\dim A$. By \ref{compare} it suffices to show $\{ \reg L^{I}(\Omega^r(M) \}_{r \geq 0}$ is bounded.
When $d = 1$ the result holds in general, see \ref{dim1}.

Assume $d \geq 2$ and the result holds when $\dim A = d -1$. Let $x$ be $A\oplus \Omega^i(M)$-superficial for $i \geq 0$. Set $B = A/(x)$, $J = I/(x)$ and $N = M/xM$.
By induction hypotheses there exists $\alpha$ such that $\reg L^J(\Omega^i(N)) \leq \alpha$ for all $i \geq 0$.
It follows that $\reg^{\geq 1} L^I(\Omega^i(M)) \leq \alpha$ for all $i \geq 0$, see \ref{reg1}.  Assume $\V_A(I^t) = \V_A(I^s)$ for all $t \geq s$. Fix $m > \max \{ \alpha, s \}$. By \ref{first}
there exists $\beta$ such that $$\e H^0(L^{I^m}(\Omega^i(M))) \leq \beta \ \text{for all}\ i \geq 0. $$
It follows from \ref{h0}   that
 $$\e H^0(L^{I}(\Omega^i(M))) \leq m(\beta + 2) \ \text{for all}\ i \geq 0. $$
So we have
 $$\reg L^I(\Omega^i(M)) \leq  m(\beta + 2) \quad \text{ for all $i \geq 0$}.$$
 The result follows.
\end{proof}

\section{Proof of Theorem \ref{special-cx1}}

We first prove the following:
\begin{theorem}
\label{first-cx1}Let $(A,\m)$ be a complete intersection of dimension $d \geq 1$ and codimension $c \geq 1$. Let $I$ be an $\m$-primary ideal and let $M$ be an MCM $A$-module with
 $\dim \V_A^\infty(I) \cap \V_A(M) \leq 1$
Let $\V(I^m) = \V(I^{s})$ for $m \geq s$. Fix $m \geq s$. Then the set
$\{ \reg L^{I^m}(\Omega^r(M) \}_{r \geq 0}$  is bounded.
\end{theorem}
\begin{proof}
Set $M_i = \Om^i(M)$. Set $J = I^m$.
By \ref{bc} we may assume that $A$ is complete and the residue field of $A$ is uncountable and algebraically closed.

Notice $\V_A^\infty(I) \cap \V_A(M)$ is a finite union of lines passing through the origin.
By \ref{red-cx-thm}  we may choose a section $0 \rt D_i \rt M_{i+2} \rt M_i \rt 0$ for $i \geq t$ with $\V_A(J) \cap V_A(D_{t}) = \{ 0 \}$. Recall
$D_{i+1} = \Om^1(D_i)$ for all $i \geq t$.
So by \ref{equi}, $\Tor^A_j(D, A/J^n) = 0$ for $j > 1$ and for $n \geq 1$. It follows that for $j \geq t + 2$ we have $\ell(\Tor^A_j(M_{t + 2}, A/J^n )) = \ell(\Tor^A_j(M_t, A/J^n))$ for $n \geq 1$.
So for $j \geq  t +2$ we have a short exact sequence of $\R(J)$-modules
$$ 0 \rt L^J(D_j) \rt L^J(M_{j + 2}) \rt L^J(M_j) \rt  0. $$
By \ref{first} we have $\reg L^J(D_j)$  for $j \geq t +2$ is bounded. By an argument similar to the proof of \ref{first} it follows that $\reg L^J(M_j)$ is bounded.
\end{proof}

The following result is needed in the proof of Theorem \ref{special-cx1}.
\begin{lemma}\label{mod-cx1}
(with hypotheses as in \ref{first-cx1}). Further assume that $A$ is complete and  the residue field of  $A$ is uncountable and algebraically closed. Set $M_i = \Om^i_A(M)$. Then there exists $x$ which is $M_i \oplus A$-superficial \wrt \ $I$  for all $i$ such that
$$\dim \V_{A/(x)}^\infty(I/(x)) \cap \V_{A/(x)}(M/xM) \leq 1.$$
\end{lemma}
\begin{proof}
Fix $n \geq 0$. Set $M_n = \Omega^n_A(M)$. Let $V = I/\m I$. We note that there is proper Zariski closed linear subspace $U_n$ of $V$ such that if the image of $x$ is in $V \setminus U$ then $x$ is $M_n \oplus A$-superficial \wrt \ $I$.

Let $\R = A[Iv]$ be the Rees algebra of $I$.
Fix $i \geq 1$. Set $$E_i = \bigoplus_{n \geq 0}\Ext^{i}_A(M, A/I^n).$$ We note that by \ref{red-rees} that $E_i$ is a finitely generated $\R$-module. So there exists a proper closed linear subspace $Z_i$ of $V$ such that if $x \in V \setminus Z_i$ then $xv$ is $E_i$-filter-regular.

As $k$ is uncountable we have
$$W =  V \setminus \left( (\bigcup_{n \geq 0}U_n )\bigcup (\bigcup_{i \geq 1} Z_i)\right) \neq \emptyset.$$
So if the image of $x$ is in $W$ then $x$ is $M_n \oplus A$-superficial \wrt \ $I$ for all $n \geq 0$ and $xv$ is $E_i$-filter regular for all $i \geq 1$.

We note that $D = \bigoplus_{i \geq 1}E_i$ is a finitely generated bigraded $ S = \R[t_1, \ldots, t_c]$ module. So $T = (0 \colon_D xv)$ is finitely generated $S$-module. Given $i$ there exists $n(i)$ such that for $n \geq n(i)$ we get $T_{in} = 0$. As $T$ is finitely generated there exists $i_0$ and $n_0$ such that $T_{in} = 0$ for all $i \geq i_0$ and $n \geq n_0$.

Let $x \in W$.
Set $A = A/(x)$ and $J = I/(x)$. Also set $N = M/xM$.
As $x$ is $A$-superficial \wrt \ $I$ we have $(I^{n+1} \colon x) = I^n$ for all $n \gg 0$; say from $n \geq m$.
For $n \geq m$ there is an exact sequence
$$ 0 \rt A/I^n \xrightarrow{\alpha_n} A/I^{n+1} \rt B/J^{n+1} \rt 0, $$
where $\alpha_n(a + I^n) = xa + I^{n+1}$.
It follows that for $i \geq i_0$ and $n \geq \max\{ m, n_0\} + 1$ there exists exact sequence
\[
0 \rt \Ext^i_A(M, A/I^n) \rt \Ext^i_A(M, A/I^{n+1}) \rt \Ext^i_A(N, B/J^{n+1}) \rt 0.
\]
So we have for $i \geq i_0$ and $n \geq \max\{ m, n_0\} + 1$;
$$ \ell(\Ext^i_A(N, B/J^{n+1})) = \ell(\Ext^i_A(M, A/I^{n+1})) - \ell(\Ext^i_A(M, A/I^{n}) ).$$

Say $\V^\infty_A(I) = \V_A(I^n)$ for $n \geq s$. Fix $n \geq \max\{s, m, n_0 \} + 1$. Then as $\V_A(M)\cap \V_A(I^n)$ has dimension $\leq 1$ it follows that there exists an isomorphism
$$\Ext^j_A(M, A/I^n) \cong \Ext^{j +2}_A(M, A/I^n)  \quad \text{for all} \ j \geq j(n).$$
 Consider $j \geq \max\{i_0 + 2, j(n), j(n+1) \}$. We have
$$ \ell(\Ext^j_A(N, B/J^{n+1})) =  \ell(\Ext^{j + 2}_A(N, B/J^{n+1})).$$
It follows that the  $\Ext^*_B(N, B/J^n)\otimes k $ has bounded Hilbert function. It follows that $\V_B(J^{n+1}) \cap \V_B(N) = \V_B(N, J^{n+1})$
has dimension $\leq 1$ over an appropriate ring of cohomological operators. The result follows.
\end{proof}
We now give
\begin{proof}[Proof of Theorem \ref{special-cx1}]
By \ref{bc} we may assume that the residue field of $A$ is uncountable and algebraically closed. The proof now follows on the same lines of proof of Theorem \ref{special} by inducting on dimension of $\dim A$ and using Theorem \ref{first-cx1}. We have to use \ref{mod-cx1} which states that there exists $x \in I$ which is $\Omega_A^i(M) \oplus A$-superficial for all $i$ \wrt \ $I$ such that $$\dim \V_{A/(x)}^\infty(I/(x)) \cap \V_{A/(x)}(M/xM) \leq 1.$$
\end{proof}

\section{Proof of Theorem \ref{cx-mcm}}
\begin{proof}[Proof of Theorem \ref{cx-mcm}]
For $ i \leq 0$ we have short exact sequences
\[
0 \rt \Om^i(M) \rt H_i \rt \Om^{i-1}(M) \rt 0 \quad \text{where $H_i$ is free $A$-module.}
\]
We note that $\V_A(\Om^i(M)) \cap \V_A^{tot}(I) = \{ 0 \}$ for all $i \leq 0$. As $\Om^{i-1}(M)$ is MCM $A$-module we get that $\Tor^A_1(\Om^{i-1}(M), A/I^n) = 0$ for all $n \geq 1$.
So we have an exact ssequence
$$ 0 \rt L^I(\Om^i(M)) \rt L^I(H_i) \rt L^I(\Om^{i-1}(M)) \rt 0.$$
Set $b_i(-) = b_i^I(-)$.
So we have $b_0(\Om^i(M)) \leq b_0(A)$ for all $i \leq 0$. Taking local cohomology we have
$$ b_1(\Om^i(M)) \leq \max \{b_1(A), b_0(\Om^{i-1}(A) \} \leq \max \{b_1(A), b_0(A) \}. $$
Again we have
$$b_2(\Om^i(M)) \leq \max \{b_2(A), b_1(\Om^{i-1}(A) \} \leq \max \{b_2(A), b_1(A), b_0(A) \}.$$
Iterating we have
$$b_j(\Om^i(M)) \leq \max \{b_j(A), \cdots, b_1(A), b_0(A) \} \quad \text{for $j = 0, \ldots, d-1$.}$$
It follows that
$$\reg L^I(\Om^i(M)) \leq d-1 + \max\{b_{d-1}(A), \cdots, b_1(A), b_0(A) \}. $$
The result follows from \ref{compare} by taking $i = 0$.
\end{proof}

\section{examples}
In this section we give bountiful examples to illustrate our results.

\s\label{av-ext} We will need the following result from \cite[10.3.8]{A-Inf}. Let $I$ be an $\m$-primary ideal in a complete intersection $R$ of codimension $t$ and dimension $d \geq 1$. Then $\m I$ is an $\m$-primary ideal of complexity $t$.

\s\label{set-ex} \emph{Setup:} Let $A$ be a complete, complete intersection of codimension $c$ and dimension $d \geq 1$. We also assume the residue field of $A$ is algebraically closed. Assume $A$ has a Noether normalization $R$. This is automatic if $A$ is equi-characteristic. In mixed characteristic the result holds when $A$ is a domain. Then $(R,\n)$ is a regular local ring of dimension $d$.

\s \label{basic} (with hypotheses as in \ref{set-ex}). Assume $d \geq r + 1$ where $r \geq 1$. Let $\bff = f_1, \ldots, f_r \in \n^2$ be a $R$-regular sequence. Then $\bff$ is an $A$-regular sequence
Set $S = R/(\bff)$ and $B = A/(\bff)$. Then $B$ is a finite flat extension of $S$. Let $J$ be an $\ov{\n}$-primary ideal in $S$. Set $I = \ov{\n}J$. Then note by \ref{av-ext} the complexity of $I^n$ is $r$ for all $n \geq 1$. Fix $n$. Let $\F$ be a minimal resolution of $I^n$ as $S$-modules. Then $\F \otimes B$ is a minimal resolution of $I^nB$. Thus $\cx_B I^nB = r$.

\s \label{mon} Let $T = k[[X_1,\ldots, X_d, Z_1, \ldots, Z_c]]/(Z_1^{a_1}, \ldots, Z_c^{a_c})$ where $a_1, \ldots, a_c \geq 2$ and $d \geq 1$. Then for any $i$ with $0 \leq i \leq c$ there exists $\m$-primary  ideals  $I_i$ in $T$ with complexity $I_i^n$ is equal to $i$ for all $n \geq 1$.

Let $\m$ be the maximal ideal of $T$. Then $\cx_T \m^n = c$ for all $n \geq 1$, see \ref{av-ext}.

We have $R = k[[X_1, \ldots, x_d]]$ is a Noether normalization of $T$ with maximal ideal $\n = (X_1, \ldots, X_d)$. Let $I$ be any $\n$-primary ideal of $R$. Then note that as $R$ is regular local the projective dimension of $I^n$ is $d$ for all $n \geq 1$.  As $T$ is flat $R$-module it follows that ideals $I^nT = (IT)^n$ in $T$ has finite projective dimension for all $n \geq 1$.

Now let $1\leq i \leq c -1$. Then $W_i = k[[X_1, \ldots, X_d, Z_1, \ldots, Z_i]]$ is a Noether normalization of $k[[X_1,\ldots, X_d, Z_1, \ldots, Z_c]]/(Z_{i+1}^{a_1}, \ldots, Z_c^{a_c})$.
It follows that $T$ is a flat extension of $S_i = W_i/(Z_1^{a_1}, \ldots, Z_{i}^{a_i})$. By an argument similar to \ref{basic} there exists ideals $I$ in $T$ such that  $\cx_T I^n = i$ in $T$ for all $n \geq 1$.
\begin{remark}
  The constructions \ref{basic} and \ref{mon} can be combined to construct bountiful examples of complete intersections $B$ having ideals $I^n$ such that $\cx_B I^n < \codim B$ for all $n$.
  However it is not known that if $A$ is a complete intersection of codimension $c$ then for any $i$ with $1 \leq i \leq c -1$ there exists ideals $I_i$ such that $\cx_A I_i^n = i$ for all $n \geq 1$.
\end{remark}

In the next examples we take support varieties in $\Pb^{c-1}(k)$, see \ref{proj}.
\begin{example}\label{e1}
(with setup as in \ref{set-ex}). Suppose $I$ be an $\m$-primary ideal  in $A$ such that $\cx I^n = r$ for all $n \geq 1$. Note $X = \Zc_A^\infty(I)$  projective set corresponding to $ \V^\infty_A(I)$ is a $r - 1$-dimensional variety in $\Pb^{c-1}$.
We note that for each algebraic set $Y$ in $\Pb^{c-1}$ there exists a MCM $A$-module $M_Y$ such that $\Zc_A(M_Y) = Y$, see \cite[2.3]{B}.
 The following examples yield  examples of MCM modules $M$ with $\Zc^\infty_A(I) \cap \Zc_A(M) = \emptyset$ (equivalently $\V_A^\infty(I)\cap V_A(M) = \{ 0 \}$).
\begin{enumerate}
  \item If $c > r$ then there are $|k|$ points in $\Pb^{c-1} \setminus X$. For each of these points $a$ there exists an MCM $A$-module $M_a$ of with support variety $\{a \}$. Note $\cx M_a = 1$

  \item If $c > r + 1$ then there are $|k|$ lines in $\Pb^{c-1} \setminus X$, see \ref{ag-pt}. For each of these lines $\ell$ there exists an MCM $A$-module $M_\ell$ of with support variety $\ell$. Note $\cx M_\ell = 2$.

  \item If $c > r + 2$ then there are $|k|$ planes in $\Pb^{c-1} \setminus X$, see \ref{ag-plane}. For each of these planes $K$ there exists an MCM $A$-module $M_K$ of with support variety $K$. Note $\cx M_K = 3$.
  \item If $r = 1$ then a general hyperplane $H$ of $\Pb^{c-1}$ will not contain any of the points in $\V^\infty_A(I)$. Let $M_H$ be the MCM module with support variety $H$.
\end{enumerate}
 The following examples yield  examples of MCM modules $M$ with $\dim \Zc^\infty_A(I) \cap \Zc_A(M) = 0$ (equivalently $\dim \V_A^\infty(I)\cap V_A(M) = 1$).
\begin{enumerate}
  \item If $c > r$ then there are $|k|$ points in $\Pb^{c-1} \setminus X$. Fix a point $p \in X$. For each of the points $a \in \Pb^{c-1} \setminus X$ let $\ell_a$ be a line joining $a$ to point
$p$. Let $M_{\ell_a}$ be the MCM $A$-module with support variety $\ell_a$. Then note $\Zc^\infty_A(I) \cap \Zc_A(M_{\ell_a})$ is a finite set of points.
\item If $r = 2$ then a general hyperplane $H$ will intersect $X$ in only finitely many points. Let $M_H$ be the MCM module with support variety $H$.
\end{enumerate}
\end{example}
\section{Proof of Theorem \ref{sunif}}
In this section we give
\begin{proof}[Proof of Theorem \ref{sunif}]
We may assume that the residue field of $A$ is infinite. Set $M_i = \Omega^i_A(M)$ for $i \geq 0$. For $i \geq 1$ consider the filtration $\Hc(i) = \{ I^n\F_{i-1} \cap M_i \}_{n \geq 0}$. We note that $\Hc(i)$ is an $I$-stable filtration on $M_i$ and we have an exact sequence of $G_I(A)$-modules:
\[
0 \rt G_{\Hc(i)}(M_i) \rt G_I(F_{i-1}) \rt G_I(M_{i-1}) \rt 0.
\]
It follows that there exists $h$ such that $\reg G_{\Hc(i)}(M_i) \leq h$ for all $i \geq 1$.
We show that for all $n \geq h + 1$ we have $I^{n+1}\F_{i-1} \cap M_i = I(I^n\F_i \cap M_i)$. This will prove the result.

Fix $i \geq 1$. Set $N = M_i$ and $\Hc = \Hc(i)$ and $G = G_I(A)$. Also set $N_n = \Hc_n = I^n\F_{i-1} \cap N$ for $n \geq 0$. Let $x_1^*, \ldots, x_d^* \in  G_1$ be $G_\Hc(N)$-filter regular. By \cite[18.3.11]{BS}   we get that
\[
\reg G_\Hc(N)/\mathbf{x}^* G_\Hc(N) \leq h.
\]
So for $n \geq h+1$ we have $N_{n} = IN_{n-1} + N_{n+1}$.
But $N_{n+1} = IN_{n} + N_{n+2}$. So $N_n = IN_{n-1} + N_{n+2}$. Iterating we get that for all $r \geq n + 1$ we have $N_n = IN_{n-1} + N_r$. But for large $r$ we have $N_r = IN_{r-1}$. So we get
$N_n = IN_n$. The result follows.
\end{proof}
\section{Appendix}
In this section we prove some results in elementary algebraic geometry that we need. We believe these results are already known. However we do not have a reference. So we give proofs.
Throughout $k$ is an algebraically closed field, $X$ is a closed subvariety of $\Pb^n(k)$. We assume $\dim X = r$.
\begin{theorem}
\label{ag-pt}(with hypotheses as above)  If $n \geq r + 2$ then there exists $|k|$-distinct lines $L$ in $\Pb^n(k)$ such that  $L \cap X = \emptyset$.
\end{theorem}
\begin{proof}
Let $p \in \Pb^n \setminus X$. Consider the projection $\pi \colon \Pb^n \setminus \{p\} \rt \Pb^{n-1}$. We have $\pi$ is a morphism of quasi-projective varieties. It is elementary to note that $\dim \pi(X) \leq \dim X$ (use the fact that if $V$ is an irreducible component of $X$ then  $\pi_V \colon V \rt \pi(V)$ is a dominant map).
As $n-1 \geq r +1$ the set $\pi(X)$ is a proper subvariety of $\Pb^{n-1}$. So $\Pb^{n-1} \setminus \pi(X)$ has cardinality $|k|$. The map $\pi$ is surjective and the inverse image of each point in $\Pb^{n-1} \setminus \pi(X)$ is a line in $\Pb^n$ not intersecting $X$.
\end{proof}
Next we show
\begin{theorem}
\label{ag-plane}(with hypotheses as above)  If $n \geq r + 3$ then there exists $|k|$-distinct planes $L$ in $\Pb^n(k)$ such that  $L \cap X = \emptyset$.
\end{theorem}
\begin{proof}
By \ref{ag-pt} we can choose a line $\ell$ in $\Pb^n$ not intersecting $X$. We may choose co-coordinates such that $\ell$ is given by $x_2 = x_3 = \cdots = x_n = 0$. Consider  $M = \Pb^{n-2}$ defined by $x_0 = x_1 = 0$. Consider the projection map $\pi \colon \Pb^n \setminus \ell \rt M = \Pb^{n-2}$.
 Recall $\pi(q) = \Sp(\ell \cup \{q\}) \cap M$.  Note $\pi$ is a morphism of quasi-projective varieties, see \cite[exercise M, p. 103.]{M}.
By an argument similar to that in \ref{ag-pt} we get that $\dim \pi(X) \leq r$. As $n-2 \geq r +1$ it follows that $\pi(X) $ is a proper subvariety in $M$. So $\Pb^{n-1} \setminus \pi(X)$ has cardinality $|k|$. The map $\pi$ is surjective. Let  $$a \in \left(\Pb^{n-1} \setminus \pi(X)  \right).$$
 Then  $E_a = \Sp(\ell \cup \{a \})$ is a plane in $\Pb^n$. It is easily verified that  $E_a \cap X = \emptyset$. It is evident that for distinct $a$ the planes $E_a$ are distinct. The result follows.
\end{proof}

\end{document}